\RequirePackage{fix-cm}
\documentclass{svjour3}
\smartqed
\usepackage{graphicx}
\usepackage{mathptmx}

\usepackage{amssymb}
\usepackage{stmaryrd}
\usepackage{mathrsfs}
\usepackage{mathtools}
\usepackage{psfrag}
\usepackage{hyperref}

\newcommand{\LAP}{{\Delta}}
\newcommand{\Sides}{\mathscr{S}}
\newcommand{\Ne}{\mathcal{N}}
\newcommand{\E}{\mathscr{E}}
\newcommand{\R}{\mathbb{R}}
\newcommand{\ysf}{\mathsf{y}}
\newcommand{\usf}{\mathsf{u}}
\newcommand{\fsf}{\mathsf{f}}
\newcommand{\vsf}{\mathsf{v}}
\newcommand{\asf}{\mathsf{a}}
\newcommand{\bbsf}{\mathsf{b}}
\newcommand{\psf}{\mathsf{p}}
\newcommand{\T}{\mathscr{T}}
\newcommand{\wsf}{\mathsf{w}}
\newcommand{\V}{\mathbb{V}}
\newcommand{\Uad}{{\mathbb{U}_{\textup{ad}}}}

\DeclareMathOperator{\diam}{diam}
\DeclareMathOperator{\osc}{osc}

\begin{document}

\title{Maximum--norm a posteriori error estimates for an optimal control problem}

\author{Alejandro Allendes
\and
Enrique Ot\'arola
\and\\
Richard Rankin
\and
Abner J. Salgado
}

\institute{
Alejandro Allendes \at
Departamento de Matem\'atica, Universidad T\'ecnica Federico Santa Mar\'ia, Valpara\'iso, Chile \\
\email{alejandro.allendes@usm.cl}
\and
Enrique Ot\'arola \at
Departamento de Matem\'atica, Universidad T\'ecnica Federico Santa Mar\'ia, Valpara\'iso, Chile \\
\email{enrique.otarola@usm.cl}
\and
Richard Rankin \at
School of Mathematical Sciences, The University of Nottingham Ningbo China, Ningbo, China \\
\email{richard.rankin@nottingham.edu.cn}
\and
Abner J. Salgado \at
Department of Mathematics, University of Tennessee, Knoxville, TN 37996, USA \\
\email{asalgad1@utk.edu}
\and
}

\date{Received: date / Accepted: date}
% The correct dates will be entered by the editor

\maketitle

\begin{abstract}
We analyze a reliable and efficient max-norm a posteriori error estimator for a control-constrained, linear--quadratic optimal control problem. The estimator yields optimal experimental rates of convergence within an adaptive loop.
\keywords{Linear--quadratic optimal control problem \and Finite element methods \and A posteriori error analysis \and Maximum--norm}
\subclass{49J20 \and 49M25 \and 65K10 \and 65N15 \and 65N30 \and 65N50\and 65Y20}
\end{abstract}

\section{Introduction}
\label{sec:intro}

Let $\Omega$ be an open and bounded polytope in $\R^{d}$, $d \in \{2,3\}$, with Lipschitz boundary $\partial \Omega$. Given 
$\ysf_{\Omega}\in L^{2}(\Omega)$ and 
$\lambda > 0$ we define the
cost functional
\begin{equation}
J(\ysf,\usf) = \frac{1}{2} \|\ysf-\ysf_{\Omega}\|_{L^{2}(\Omega)}^{2} + \frac{\lambda}{2} \|\usf\|_{L^{2}(\Omega)}^{2}.
\end{equation}
In this article we devise max-norm a posteriori error estimators for the following optimal control problem: Find
\begin{equation}
\label{min}
\min\, J(\ysf,\usf)
\end{equation}
subject to, for a given $\fsf \in L^2(\Omega)$, the linear elliptic partial differential equation (PDE)
\begin{equation}
\label{state_equation}
-\Delta\ysf = \fsf+\usf \quad \textrm{in}~\Omega, \qquad \ysf = 0 \quad \textrm{on}~\partial\Omega,
\end{equation}
and, for $\asf, \bbsf \in \R$ with $\asf \leq \bbsf$, the control constraints
\begin{equation}
\label{control_constraint}
\usf\in \Uad, \qquad \Uad:=\{ \vsf \in L^2(\Omega):~\asf\leq \vsf(\boldsymbol{x}) \leq \bbsf \mbox{ for almost every }\boldsymbol{x}\mbox{ in }\Omega\}.
\end{equation}

The pioneering work \cite{rhn} presented and analyzed, in two dimensions, the first max-norm a posteriori error estimator for the state equation \eqref{state_equation}. These results were later extended to more dimensions and both nonlinear and geometric problems \cite{deldia,demlow2014maximum,MR2249676}. To our knowledge, max-norm a posteriori error estimation for the optimal control problem \eqref{min}--\eqref{control_constraint} has not been considered previously in the literature. This is the novelty of our contribution.

\section{Notation}
\label{sec:notation}

Let $\T = \{T\}$ be a conforming simplicial mesh of $\bar \Omega$ \cite{MR2050138}, $h_T =\diam(T)$ and
\begin{equation}
\ell_\T = | \log ( \max_{T \in \T} h_T^{-1})|.
\end{equation}
We assume that $\T$ is a member of a shape regular family of meshes.
Define
\begin{equation*}
  \mathbb{V}(\T):=
  \{
    w \in C^{0}(\bar{\Omega}):~w_{|T}\in\mathbb{P}_{1}(T)~\forall T\in\T~\textrm{and}~w_{|\partial\Omega}=0
  \},
  \qquad
  \Uad(\T) = \Uad \cap \V(\T).
\end{equation*}
We denote by $\Sides = \{S\}$ the set of internal ($d-1$)-dimensional interelement boundaries of $\T$ and $h_S = \diam(S)$. If $T\in\T$, $\Sides_{T} \subset \Sides$ is the set of sides of $T$. For $S \in \Sides$ we set $\Ne_S = \{ T^+, T^-\}$ such that $S=T^+ \cap T^-$. For $T \in \T$, we define
\begin{equation}
  \Ne_T := \left\{ T' \in \T : \Sides_T \cap \Sides_{T'} \neq \emptyset \right\}.
  \label{eq:NeTstar}
\end{equation}
For $w_{\T} \in \V(\T)$ and $S \in \Sides$ with $\Ne_S = \{T^+, T^-\}$, the jump or interelement residual is
$
  \llbracket \nabla w_{\T} \cdot \nu  \rrbracket = \nu^+ \cdot \nabla w_{\T|T^+} + \nu^- \cdot \nabla w_{\T|T^-} ,
$
where $\nu^+, \nu^-$ are the unit normals to $S$ pointing towards $T^+$, $T^{-} \in \T$. The $L^2(\Omega)$ inner product is $(\cdot,\cdot)$. By $A \lesssim B$ we mean that $A \leq c B$ for a nonessential constant $c$ that might change at each occurrence.

\section{Optimal control problem}\label{control_optimo}

The necessary and sufficient optimality conditions of \eqref{min}--\eqref{control_constraint} read: Find $(\bar{\ysf},\bar{\psf},\bar{\usf})\in H_{0}^{1}(\Omega) \times H_{0}^{1}(\Omega) \times \Uad$ such that
\begin{equation}
\label{optimal_system_1}
\left\{
\begin{array}{l}
(\nabla\bar{\ysf}, \nabla \vsf)_{L^2(\Omega)} = (\fsf+\bar{\usf},\vsf)_{L^{2}(\Omega)}
\mbox{ }\forall~\vsf\in H_{0}^{1}(\Omega),\\
(\nabla \wsf,\nabla \bar{\psf})_{L^2(\Omega)} = (\bar{\ysf}-\ysf_{\Omega},\wsf)_{L^{2}(\Omega)} 
\mbox{ }\forall~\wsf\in H_{0}^{1}(\Omega),\\
\left(\bar{\psf}+\lambda\bar{\usf},\usf-\bar{\usf}\right)_{L^{2}(\Omega)}\geq 0
\mbox{ }\forall~\usf\in\mathbb{U}_{ad}.
\end{array}
\right.
\end{equation}
Following \cite{JK:95}, there exists $r>d$ such that $\bar{\ysf}, \bar{\psf} \in W^{1,r}(\Omega) \hookrightarrow C(\bar\Omega)$. Hence, $\bar{\usf} \in C(\bar \Omega) \cap H^1_0(\Omega)$ since
\begin{equation}\label{eq:u}
  \bar{\usf}=\Pi\left(-\lambda^{-1}\bar{\psf} \right), \qquad 
  \Pi\left(w\right)(\boldsymbol{x}):=\min\left\{\bbsf,\max\left\{\asf,w(\boldsymbol{x})\right\}\right\}\mbox{ for all }\boldsymbol{x}\mbox{ in }\bar\Omega.
\end{equation}
For $\mathcal{G} \subset \Omega$ this operator is nonexpansive in $L^\infty(\mathcal{G})$, i.e.,
\begin{equation}\label{Lipschitz}
\left\| \Pi (w_1)-\Pi (w_2)  \right\|_{L^\infty(\mathcal{G})}\le \left\| w_1-w_2 \right\|_{L^\infty(\mathcal{G})}
\quad \forall w_1,w_2\in C(\bar \Omega).
\end{equation}

We approximate the solution of \eqref{optimal_system_1} by finding
$(\bar{\ysf}_{\T},\bar{\psf}_{\T},\bar{\usf}_{\T})\in \mathbb{V}(\T)\times \mathbb{V}(\T)\times \Uad(\T)$ such that
\begin{equation}
\label{optimal_system_discrete_1}
\left\{
\begin{array}{l}
(\nabla \bar{\ysf}_{\T},\nabla \vsf_{\T})_{L^{2}(\Omega)}
= (\fsf+\bar{\usf}_{\T},\vsf_{\T})_{L^{2}(\Omega)}
\mbox{ }\forall\vsf_{\T}\in \mathbb{V}(\T),\\
(\nabla \wsf_{\T}, \nabla \bar{\psf}_{\T})_{L^2(\Omega)}
= (\bar{\ysf}_{\T}-\ysf_{\Omega},\wsf_{\T})_{L^{2}(\Omega)}
\mbox{ }\forall\wsf_{\T}\in \mathbb{V}(\T),\\
\left(\bar{\psf}_{\T}+\lambda\bar{\usf}_{\T},\usf_{\T}-\bar{\usf}_{\T}\right)_{L^{2}(\Omega)}\geq 0
\mbox{ }\forall\usf_{\T}\in\mathbb{U}_{ad}(\T).
\end{array}
\right.
\end{equation}

\section{A posteriori error analysis: reliability}
\label{sec:reliable}

We begin by defining the local error indicators
\begin{align}
\label{eq:defofEyloc}
 \E_{\ysf}(\bar{\ysf}_{\T},\bar{\usf}_{\T}; T) &= h_T^{2-d/2} \| \fsf+\bar{\usf}_{\T}\|_{L^{2}(T)} 
   + h_T \| \llbracket \nabla \bar{\ysf}_{\T} \cdot \nu \rrbracket  \|_{L^{\infty}(\partial T \setminus \partial \Omega)},
\\
\label{eq:defofEploc}
 \E_{\psf}(\bar{\psf}_{\T},\bar{\ysf}_{\T}; T) &= h_T^{2-d/2} \| \bar{\ysf}_{\T} - \ysf_\Omega \|_{L^{2}(T)} 
 + h_T \| \llbracket \nabla \bar{\psf}_{\T} \cdot \nu \rrbracket  \|_{L^{\infty}(\partial T \setminus \partial \Omega)},
\\
\label{eq:defofEuloc}
 \E_{\usf}(\bar{\usf}_{\T},\bar{\psf}_{\T}; T) &= \| \Pi(-\bar \psf_{\T}/ \lambda) - \bar \usf_{\T}\|_{L^\infty(T)},
\\
\label{eq:defofEypu}
  \E(\bar{\ysf}_{\T},\bar{\psf}_{\T},\bar{\usf}_{\T}; T) &= 
  \left(\E_{\ysf}^2(\bar{\ysf}_{\T},\bar{\usf}_{\T}; T)
  +
  \E_{\psf}^2(\bar{\psf}_{\T},\bar{\ysf}_{\T}; T)
  +
  \E_{\usf}^2(\bar{\usf}_{\T},\bar{\psf}_{\T}; T)\right)^{1/2},
\end{align}
and the global a posteriori error estimators
\begin{align}
\label{eq:defofEy}
\E_{\ysf}(\bar{\ysf}_{\T},\bar{\usf}_{\T}; \T)&=\max_{T \in \T} \E_{\ysf}(\bar{\ysf}_{\T},\bar{\usf}_{\T}; T),
\\
\label{eq:defofEp}
\E_{\psf}(\bar{\psf}_{\T},\bar{\ysf}_{\T}; \T)&=\max_{T \in \T} \E_{\psf}(\bar{\psf}_{\T},\bar{\ysf}_{\T}; T),
\\
\label{eq:defofEu}
\E_{\usf}(\bar{\usf}_{\T},\bar{\psf}_{\T}; \T)&=\max_{T \in \T} \E_{\usf}(\bar{\usf}_{\T},\bar{\psf}_{\T}; T),
\\
\label{eq:defofEypuglobal}
  \E(\bar{\ysf}_{\T},\bar{\psf}_{\T},\bar{\usf}_{\T}; \T) &= 
  \left(\E_{\ysf}^2(\bar{\ysf}_{\T},\bar{\usf}_{\T}; \T)
  +
  \E_{\psf}^2(\bar{\psf}_{\T},\bar{\ysf}_{\T}; \T)
  +
  \E_{\usf}^2(\bar{\usf}_{\T},\bar{\psf}_{\T}; \T)\right)^{1/2}.
\end{align}
We also introduce two auxilliary variables: $\hat{\ysf}, \hat{\psf} \in H_0^1(\Omega)$ which solve, respectively,
\begin{equation}
\label{eq:y_hat}
(\nabla \hat{\ysf}, \nabla \vsf) =(\fsf+\bar{\usf}_{\T},\vsf) \quad  \forall~ \vsf \in H_{0}^{1}(\Omega),
\quad (\nabla \wsf, \nabla \hat{\psf}) = (\bar{\ysf}_{\T}-\ysf_{\Omega},\wsf) \quad  \forall~ \wsf \in H_{0}^{1}(\Omega).
\end{equation}
We note that $\hat{\ysf}, \hat{\psf} \in H_0^1(\Omega) \cap C(\bar \Omega)$ and thus we can invoke \cite[Lemma 4.2]{AORS2} to conclude that
\begin{equation}\label{final_upper_bound_state}
\|\hat{\ysf}-\bar{\ysf}_{\T}\|_{L^\infty(\Omega)} \lesssim \ell_\T \E_{\ysf}(\bar{\ysf}_{\T},\bar{\usf}_{\T}; \T),
\qquad \|\hat{\psf}-\bar{\psf}_{\T}\|_{L^\infty(\Omega)} \lesssim \ell_\T \E_{\psf}(\bar{\psf}_{\T},\bar{\ysf}_{\T}; \T).
\end{equation}
Finally, for $e_{\bar{\ysf}}=\bar{\ysf}-\bar{\ysf}_{\T}$, $e_{\bar{\psf}}=\bar{\psf}-\bar{\psf}_{\T}$ and $e_{\bar{\usf}}=\bar{\usf}-\bar{\usf}_{\T}$ we define
\begin{equation}
\|(e_{\bar{\ysf}},e_{\bar{\psf}},e_{\bar{\usf}})\|_{\Omega}^2 :=
\|e_{\bar{\ysf}}\|_{L^\infty(\Omega)}^{2}
+
\|e_{\bar{\psf}}\|_{L^\infty(\Omega)}^{2}
+
\|e_{\bar{\usf}}\|_{L^\infty(\Omega)}^{2}
.
\end{equation}

\begin{theorem}[global reliability]\label{th:global_reliability}
Let $(\bar{\ysf},\bar{\psf},\bar{\usf}) \in H_0^1(\Omega) \times H_0^1(\Omega) \times L^2(\Omega)$ be the solution to \eqref{optimal_system_1} and $(\bar{\ysf}_{\T},\bar{\psf}_{\T},\bar{\usf}_{\T}) \in \V(\T) \times \V(\T) \times \Uad(\T)$ its numerical approximation obtained as the solution to \eqref{optimal_system_discrete_1}. Then 
\begin{equation}
\label{eq:reliability}
\|(e_{\bar{\ysf}},e_{\bar{\psf}},e_{\bar{\usf}})\|_{\Omega}
\lesssim
\max\{1,\ell_\T\}
\E(\bar{\ysf}_{\T},\bar{\psf}_{\T},\bar{\usf}_{\T}; \T).
\end{equation}
\end{theorem}
\begin{proof}
We proceed in five steps.

\noindent \framebox{Step 1.} First we control the error $\|  \bar{\usf} -\bar{\usf}_{\T} \|_{L^2(\Omega)}$. Define $\tilde{\usf} = \Pi (-\lambda^{-1}\bar{\psf}_{\T})$, which can be equivalently characterized by
 \begin{equation}
  \label{tildeu}
  (\bar{\psf}_{\T}+\lambda \tilde{\usf}, \usf - \tilde{\usf}) \geq 0 \quad \forall \usf \in \Uad.
 \end{equation}
We first bound $\|\bar{\usf}-\tilde{\usf}\|_{L^2(\Omega)}$. Set $\usf = \tilde \usf$ in \eqref{optimal_system_1},  $\usf = \bar \usf$ in \eqref{tildeu} and add the results to obtain
\[
  \lambda\| \bar{\usf} - \tilde{\usf} \|^2_{L^2(\Omega)} \leq (\bar{\psf} - \bar{\psf}_{\T}, \tilde{\usf} - \bar{\usf}).
\]
To bound the right hand side of the previous expression, we let $(\tilde{\ysf},\tilde{\psf}) \in H_0^1(\Omega) \times H_0^1(\Omega)$ be such that
\begin{eqnarray*}
(\nabla \tilde{\ysf}, \nabla \vsf) = (\fsf+\tilde{\usf},\vsf)\quad \forall~ \vsf \in H_{0}^{1}(\Omega),
\qquad
(\nabla \tilde{\psf}, \nabla \wsf) = (\tilde{\ysf}-\ysf_{\Omega},\wsf) \quad \forall~ \wsf \in H_{0}^{1}(\Omega).
\end{eqnarray*}
With the auxilliary adjoint state $\tilde{\psf}$ at hand, we thus arrive at
\begin{equation}
\lambda\| \bar{\usf} - \tilde{\usf} \|^2_{L^2(\Omega)}
\leq 
(\bar{\psf}  - \tilde{\psf}, \tilde{\usf} - \bar{\usf}) + 
(\tilde{\psf}  - \hat{\psf}, \tilde{\usf} - \bar{\usf}) +
(\hat{\psf} - \bar{\psf}_{\T}, \tilde{\usf} - \bar{\usf}).
\label{eq:control_L21}
\end{equation}
We now observe that
$(\nabla (\tilde{\ysf}-\bar{\ysf}),\nabla \vsf) = (\tilde{\usf}-\bar{\usf},\vsf)$ for all $\vsf \in H_{0}^{1}(\Omega)$ and 
$(\nabla \wsf,\nabla(\bar{\psf}-\tilde{\psf}))=(\bar{\ysf}-\tilde{\ysf},\wsf)$ $\forall \wsf \in H_{0}^{1}(\Omega)$. Hence,
\[
 (\bar{\psf}  - \tilde{\psf}, \tilde{\usf} - \bar{\usf}) = (\nabla(\tilde{\ysf}-\bar{\ysf}),\nabla(\bar{\psf}  - \tilde{\psf})) 
 = -\|\bar{\ysf}-\tilde{\ysf}\|_{L^{2}(\Omega)}^{2}\leq 0.
\]
In view of this, an application of the Cauchy--Schwarz and Young's inequalities to \eqref{eq:control_L21} yields
\begin{equation}\label{eq:usf-tildeusf_3}
\| \bar{\usf} - \tilde{\usf} \|^2_{L^2(\Omega)} \lesssim
\|\tilde{\psf}  - \hat{\psf}\|_{L^{2}(\Omega)}^{2} +
\|\hat{\psf} - \bar{\psf}_{\T}\|_{L^{2}(\Omega)}^{2}.
\end{equation}
We now control $\| \tilde{\psf}  - \hat{\psf}\|_{L^{2}(\Omega)}^{2}$. Since $(\nabla \wsf, \nabla(\tilde{\psf}  - \hat{\psf}))=(\tilde{\ysf}-\bar{\ysf}_{\T},\wsf)$, for all $\wsf\in H_{0}^{1}(\Omega)$, we have that
\begin{equation*}
\begin{split}
\|\tilde{\psf}  - \hat{\psf}\|_{L^{2}(\Omega)}^{2} 
\lesssim &
(\nabla(\tilde{\psf}  - \hat{\psf}),\nabla(\tilde{\psf}  - \hat{\psf}))
\\
= & (\tilde{\ysf}-\bar{\ysf}_{\T},\tilde{\psf}  - \hat{\psf})
\lesssim \left(\|\tilde{\ysf}-\hat{\ysf}\|_{L^{2}(\Omega)}+\|\hat{\ysf}-\bar{\ysf}_{\T}\|_{L^{2}(\Omega)}\right)
\|\tilde{\psf}  - \hat{\psf}\|_{L^{2}(\Omega)}.
\end{split}
\end{equation*}
Consequently,
$
\|\tilde{\psf}  - \hat{\psf}\|_{L^{2}(\Omega)} 
\lesssim \|\tilde{\ysf}-\hat{\ysf}\|_{L^{2}(\Omega)}+\|\hat{\ysf}-\bar{\ysf}_{\T}\|_{L^{2}(\Omega)}
$.
This, in conjunction with \eqref{eq:usf-tildeusf_3}, yields
\begin{equation}\label{eq:usf-tildeusf_4}
\| \bar{\usf} - \tilde{\usf} \|^2_{L^2(\Omega)} 
\lesssim
\|\tilde{\ysf}-\hat{\ysf}\|_{L^{2}(\Omega)}^2
+\|\hat{\ysf}-\bar{\ysf}_{\T}\|_{L^{2}(\Omega)}^2
+\| \hat{\psf} - \bar{\psf}_{\T}\|_{L^{2}(\Omega)}^{2}.
\end{equation}
Upon observing that $(\nabla (\tilde{\ysf}  - \hat{\ysf}),\nabla \vsf)=(\tilde{\usf}-\bar{\usf}_{\T},\vsf)$, for all $\vsf\in H_{0}^{1}(\Omega)$, we obtain that
$$
\|\tilde{\ysf}  - \hat{\ysf}\|_{L^{2}(\Omega)}^{2} 
\lesssim ( \nabla(\tilde{\ysf}  - \hat{\ysf}),\nabla(\tilde{\ysf}  - \hat{\ysf}))=(\tilde{\usf}-\bar{\usf}_{\T},\tilde{\ysf}  - \hat{\ysf})
\leq \|\tilde{\usf}-\bar{\usf}_{\T}\|_{L^{2}(\Omega)}\|\tilde{\ysf}  - \hat{\ysf}\|_{L^{2}(\Omega)}
$$
and hence, 
$
\|\tilde{\ysf}  - \hat{\ysf}\|_{L^{2}(\Omega)} 
\lesssim
\|\tilde{\usf}-\bar{\usf}_{\T}\|_{L^{2}(\Omega)}
$.
Combining this with \eqref{eq:usf-tildeusf_4} implies that
$$
\| \bar{\usf} - \tilde{\usf} \|^2_{L^2(\Omega)} 
\lesssim
\|\tilde{\usf}-\bar{\usf}_{\T}\|_{L^{2}(\Omega)}^2
+\|\hat{\ysf}-\bar{\ysf}_{\T}\|_{L^{2}(\Omega)}^2
+\| \hat{\psf} - \bar{\psf}_{\T}\|_{L^{2}(\Omega)}^{2}.
$$
The triangle inequality then allows us to conclude that
\begin{align}
\| \bar{\usf} - \bar{\usf}_{\T} \|^2_{L^2(\Omega)} 
\lesssim &
\|\tilde{\usf}-\bar{\usf}_{\T}\|_{L^{2}(\Omega)}^2
+\|\hat{\ysf}-\bar{\ysf}_{\T}\|_{L^{2}(\Omega)}^2
+\| \hat{\psf} - \bar{\psf}_{\T}\|_{L^{2}(\Omega)}^{2}\nonumber
\\
\lesssim &
\|\tilde{\usf}-\bar{\usf}_{\T}\|_{L^{\infty}(\Omega)}^2
+
\|\hat{\ysf}-\bar{\ysf}_{\T}\|_{L^{\infty}(\Omega)}^2
+
\| \hat{\psf} - \bar{\psf}_{\T}\|_{L^{\infty}(\Omega)}^{2}.\label{error_u}
\end{align}
%%%%%%%%%%%%%%%%%%%%%%%%%%%%%%%%%%%%%%%%%%%%%%%%%%%%%%%%%%%%%%%
%%%%%%%%%%%%%%%%%%%%%%%%%%%%%%%%%%%%%%%%%%%%%%%%%%%%%%%%%%%%%%%
%%%%%%%%%%%%%%%%%%%%%%%%%%%%%%%%%%%%%%%%%%%%%%%%%%%%%%%%%%%%%%%
%%%%%%%%%%%%%%%%%%%%%%%%%%%%%%%%%%%%%%%%%%%%%%%%%%%%%%%%%%%%%%%
%%%%%%%%%%%%%%%%%%%%%%%%%%%%%%%%%%%%%%%%%%%%%%%%%%%%%%%%%%%%%%%
%%%%%%%%%%%%%%%%%%%%%%%%%%%%%%%%%%%%%%%%%%%%%%%%%%%%%%%%%%%%%%%
%%%%%%%%%%%%%%%%%%%%%%%%%%%%%%%%%%%%%%%%%%%%%%%%%%%%%%%%%%%%%%%
%%%%%%%%%%%%%%%%%%%%%%%%%%%%%%%%%%%%%%%%%%%%%%%%%%%%%%%%%%%%%%%
%%%%%%%%%%%%%%%%%%%%%%%%%%%%%%%%%%%%%%%%%%%%%%%%%%%%%%%%%%%%%%%
%%%%%%%%%%%%%%%%%%%%%%%%%%%%%%%%%%%%%%%%%%%%%%%%%%%%%%%%%%%%%%%

\noindent \framebox{Step 2.} In this step we control $\|\bar{\ysf}-\bar{\ysf}_{\T}\|_{L^\infty(\Omega)}$. 
In view of the results of \cite{JK:95} there exists $r>d$ such that
\[
  \|\bar{\ysf}-\hat{\ysf}\|_{L^\infty(\Omega)}
  \lesssim 
  \|\bar{\ysf}-\hat{\ysf}\|_{W^{1,r}(\Omega)}
  \lesssim
  \|\bar{\usf}-\bar{\usf}_{\T}\|_{L^{2}(\Omega)}.
\]
Consequently, the triangle inequality and \eqref{error_u} give us that
\begin{equation}\label{error_y}
\|\bar{\ysf}-\bar{\ysf}_{\T}\|_{L^\infty(\Omega)}^{2}
\lesssim
\|\hat{\ysf}-\bar{\ysf}_{\T}\|_{L^{\infty}(\Omega)}^2
+\| \hat{\psf} - \bar{\psf}_{\T}\|_{L^{\infty}(\Omega)}^{2}
+\|\tilde{\usf}-\bar{\usf}_{\T}\|_{L^{\infty}(\Omega)}^2.
\end{equation}
%%%%%%%%%%%%%%%%%%%%%%%%%%%%%%%%%%%%%%%%%%%%%%%%%%%%%%%%%%%%%%%
%%%%%%%%%%%%%%%%%%%%%%%%%%%%%%%%%%%%%%%%%%%%%%%%%%%%%%%%%%%%%%%
%%%%%%%%%%%%%%%%%%%%%%%%%%%%%%%%%%%%%%%%%%%%%%%%%%%%%%%%%%%%%%%
%%%%%%%%%%%%%%%%%%%%%%%%%%%%%%%%%%%%%%%%%%%%%%%%%%%%%%%%%%%%%%%
%%%%%%%%%%%%%%%%%%%%%%%%%%%%%%%%%%%%%%%%%%%%%%%%%%%%%%%%%%%%%%%
%%%%%%%%%%%%%%%%%%%%%%%%%%%%%%%%%%%%%%%%%%%%%%%%%%%%%%%%%%%%%%%
%%%%%%%%%%%%%%%%%%%%%%%%%%%%%%%%%%%%%%%%%%%%%%%%%%%%%%%%%%%%%%%
%%%%%%%%%%%%%%%%%%%%%%%%%%%%%%%%%%%%%%%%%%%%%%%%%%%%%%%%%%%%%%%
%%%%%%%%%%%%%%%%%%%%%%%%%%%%%%%%%%%%%%%%%%%%%%%%%%%%%%%%%%%%%%%
%%%%%%%%%%%%%%%%%%%%%%%%%%%%%%%%%%%%%%%%%%%%%%%%%%%%%%%%%%%%%%%

\noindent \framebox{Step 3.} To bound $\|\bar{\psf}-\bar{\psf}_{\T}\|_{L^\infty(\Omega)}$ 
we again use \cite{JK:95}
to conclude that there exists $r>d$ such that
\[
  \|\bar{\psf}-\hat{\psf}\|_{L^\infty(\Omega)}
  \lesssim 
  \|\bar{\psf}-\hat{\psf}\|_{W^{1,r}(\Omega)}
  \lesssim
  \|\bar{\ysf}-\bar{\ysf}_{\T}\|_{L^{2}(\Omega)}
  \lesssim
  \|\bar{\ysf}-\bar{\ysf}_{\T}\|_{L^{\infty}(\Omega)}.
\]
Thus, this estimate and \eqref{error_y} imply that
\begin{align}
  \|\bar{\psf}-\bar{\psf}_{\T}\|_{L^\infty(\Omega)}^{2}
  \lesssim &
  \|\bar{\psf}-\hat{\psf}\|_{L^\infty(\Omega)}^{2}+\|\hat{\psf}-\bar{\psf}_{\T}\|_{L^\infty(\Omega)}^{2}\nonumber
  \\
  \lesssim &
  \|\hat{\ysf}-\bar{\ysf}_{\T}\|_{L^{\infty}(\Omega)}^2
  +\| \hat{\psf} - \bar{\psf}_{\T}\|_{L^{\infty}(\Omega)}^{2}
  +\|\tilde{\usf}-\bar{\usf}_{\T}\|_{L^{\infty}(\Omega)}^2.\label{error_p}
  \end{align}
%%%%%%%%%%%%%%%%%%%%%%%%%%%%%%%%%%%%%%%%%%%%%%%%%%%%%%%%%%%%%%%
%%%%%%%%%%%%%%%%%%%%%%%%%%%%%%%%%%%%%%%%%%%%%%%%%%%%%%%%%%%%%%%
%%%%%%%%%%%%%%%%%%%%%%%%%%%%%%%%%%%%%%%%%%%%%%%%%%%%%%%%%%%%%%%
%%%%%%%%%%%%%%%%%%%%%%%%%%%%%%%%%%%%%%%%%%%%%%%%%%%%%%%%%%%%%%%
%%%%%%%%%%%%%%%%%%%%%%%%%%%%%%%%%%%%%%%%%%%%%%%%%%%%%%%%%%%%%%%
%%%%%%%%%%%%%%%%%%%%%%%%%%%%%%%%%%%%%%%%%%%%%%%%%%%%%%%%%%%%%%%
%%%%%%%%%%%%%%%%%%%%%%%%%%%%%%%%%%%%%%%%%%%%%%%%%%%%%%%%%%%%%%%
%%%%%%%%%%%%%%%%%%%%%%%%%%%%%%%%%%%%%%%%%%%%%%%%%%%%%%%%%%%%%%%
%%%%%%%%%%%%%%%%%%%%%%%%%%%%%%%%%%%%%%%%%%%%%%%%%%%%%%%%%%%%%%%
%%%%%%%%%%%%%%%%%%%%%%%%%%%%%%%%%%%%%%%%%%%%%%%%%%%%%%%%%%%%%%%

\noindent \framebox{Step 4.} 
The goal of this step is to control the error $\|  \bar{\usf} -\bar{\usf}_{\T} \|_{L^\infty(\Omega)}$. We begin with the basic estimate
\begin{equation}
 \label{u-uh-utilde-inf}
 \|  \bar{\usf} -\bar{\usf}_{\T} \|_{L^\infty(\Omega)} \leq  \|\bar{\usf}-\tilde{\usf}\|_{L^\infty(\Omega)} +\|\tilde{\usf}-\bar{\usf}_{\T}\|_{L^\infty(\Omega)} .
\end{equation}
Using \eqref{Lipschitz} we have that
$
\|\bar{\usf}-\tilde{\usf}\|_{L^\infty(\Omega)}
=
\|\Pi (-\lambda^{-1}\bar{\psf})-\Pi (-\lambda^{-1} \bar{\psf}_{\T})\|_{L^\infty(\Omega)}
\lesssim
\|\bar{\psf}-\bar{\psf}_{\T}\|_{L^\infty(\Omega)}.
$
Therefore, upon combining this with \eqref{u-uh-utilde-inf} and \eqref{error_p}, we can conclude that
\begin{equation}\label{error_u_infty}
\|  \bar{\usf} -\bar{\usf}_{\T} \|_{L^\infty(\Omega)}^{2}
\lesssim
\|\hat{\ysf}-\bar{\ysf}_{\T}\|_{L^{\infty}(\Omega)}^2
+\| \hat{\psf} - \bar{\psf}_{\T}\|_{L^{\infty}(\Omega)}^{2}
+\|\tilde{\usf}-\bar{\usf}_{\T}\|_{L^{\infty}(\Omega)}^2.
\end{equation}
%%%%%%%%%%%%%%%%%%%%%%%%%%%%%%%%%%%%%%%%%%%%%%%%%%%%%%%%%%%%%%%
%%%%%%%%%%%%%%%%%%%%%%%%%%%%%%%%%%%%%%%%%%%%%%%%%%%%%%%%%%%%%%%
%%%%%%%%%%%%%%%%%%%%%%%%%%%%%%%%%%%%%%%%%%%%%%%%%%%%%%%%%%%%%%%
%%%%%%%%%%%%%%%%%%%%%%%%%%%%%%%%%%%%%%%%%%%%%%%%%%%%%%%%%%%%%%%
%%%%%%%%%%%%%%%%%%%%%%%%%%%%%%%%%%%%%%%%%%%%%%%%%%%%%%%%%%%%%%%
%%%%%%%%%%%%%%%%%%%%%%%%%%%%%%%%%%%%%%%%%%%%%%%%%%%%%%%%%%%%%%%
%%%%%%%%%%%%%%%%%%%%%%%%%%%%%%%%%%%%%%%%%%%%%%%%%%%%%%%%%%%%%%%
%%%%%%%%%%%%%%%%%%%%%%%%%%%%%%%%%%%%%%%%%%%%%%%%%%%%%%%%%%%%%%%
%%%%%%%%%%%%%%%%%%%%%%%%%%%%%%%%%%%%%%%%%%%%%%%%%%%%%%%%%%%%%%%
%%%%%%%%%%%%%%%%%%%%%%%%%%%%%%%%%%%%%%%%%%%%%%%%%%%%%%%%%%%%%%%

\noindent \framebox{Step 5.} 
The claimed result follows upon gathering \eqref{error_y}, \eqref{error_p} and \eqref{error_u_infty}, and using \eqref{eq:defofEu} and \eqref{final_upper_bound_state}.
\end{proof}
%%%%%%%%%%%%%%%%%%%%%%%%%%%%%%%%%%%%%%%%%%%%%%%%%%%%%%%%%%%%%%
%%%%%%%%%%%%%%%%%%%%%%%%%%%%%%%%%%%%%%%%%%%%%%%%%%%%%%%%%%%%%%
%%%%%%%%%%%%%%%%%%%%%%%%%%%%%%%%%%%%%%%%%%%%%%%%%%%%%%%%%%%%%%
%%%%%%%%%%%%%%%%%%%%%%%%%%%%%%%%%%%%%%%%%%%%%%%%%%%%%%%%%%%%%%
%%%%%%%%%%%%%%%%%%%%%%%%%%%%%%%%%%%%%%%%%%%%%%%%%%%%%%%%%%%%%%
%%%%%%%%%%%%%%%%%%%%%%%%%%%%%%%%%%%%%%%%%%%%%%%%%%%%%%%%%%%%%%
%%%%%%%%%%%%%%%%%%%%%%%%%%%%%%%%%%%%%%%%%%%%%%%%%%%%%%%%%%%%%%
%%%%%%%%%%%%%%%%%%%%%%%%%%%%%%%%%%%%%%%%%%%%%%%%%%%%%%%%%%%%%%
%%%%%%%%%%%%%%%%%%%%%%%%%%%%%%%%%%%%%%%%%%%%%%%%%%%%%%%%%%%%%%
%%%%%%%%%%%%%%%%%%%%%%%%%%%%%%%%%%%%%%%%%%%%%%%%%%%%%%%%%%%%%%
%%%%%%%%%%%%%%%%%%%%%%%%%%%%%%%%%%%%%%%%%%%%%%%%%%%%%%%%%%%%%%%
%%%%%%%%%%%%%%%%%%%%%%%%%%%%%%%%%%%%%%%%%%%%%%%%%%%%%%%%%%%%%%
%%%%%%%%%%%%%%%%%%%%%%%%%%%%%%%%%%%%%%%%%%%%%%%%%%%%%%%%%%%%%%
%%%%%%%%%%%%%%%%%%%%%%%%%%%%%%%%%%%%%%%%%%%%%%%%%%%%%%%%%%%%%%
%%%%%%%%%%%%%%%%%%%%%%%%%%%%%%%%%%%%%%%%%%%%%%%%%%%%%%%%%%%%%%
%%%%%%%%%%%%%%%%%%%%%%%%%%%%%%%%%%%%%%%%%%%%%%%%%%%%%%%%%%%%%%
%%%%%%%%%%%%%%%%%%%%%%%%%%%%%%%%%%%%%%%%%%%%%%%%%%%%%%%%%%%%%%
%%%%%%%%%%%%%%%%%%%%%%%%%%%%%%%%%%%%%%%%%%%%%%%%%%%%%%%%%%%%%%
%%%%%%%%%%%%%%%%%%%%%%%%%%%%%%%%%%%%%%%%%%%%%%%%%%%%%%%%%%%%%%
%%%%%%%%%%%%%%%%%%%%%%%%%%%%%%%%%%%%%%%%%%%%%%%%%%%%%%%%%%%%%%
\section{A posteriori error analysis: efficiency}
\label{subsub:efficient}

Let $\mathcal{P}_{\T}$ denote the $L^2$-projection onto piecewise linear, over $\T$, functions. For $g \in L^2(\Omega)$ and $\mathcal{M} \subset \T$ we define 
\begin{equation}
\label{eq:osc_localf}
 \osc_{\T}(g;\mathcal{M})^2 =  \sum_{T \in \mathcal{M}} h^{4-d}_{T} \| g - \mathcal{P}_{\T}g \|_{L^{2}(T)}^2 .
\end{equation}

\begin{lemma}[local efficiency of $\E_{\ysf}$] 
In the setting of Theorem \ref{th:global_reliability} we have that
\begin{equation}
\label{eq:efficiencyy}
\E_{\ysf} (\bar{\ysf}_{\T},\bar{\usf}_{\T}; T) \lesssim \|  \bar{\ysf} - \bar{\ysf}_{\T} \|_{L^{\infty}(\Ne_T)} + h_{T}^{2} \| \bar{\usf} - \bar{\usf}_\T \|_{L^\infty(\Ne_T)} 
+ \osc_{\T}(\fsf;\Ne_T),
\end{equation}
where the hidden constant is independent of the size of the elements in the mesh $\T$ and $\#\T$. 
\label{le:local_eff_y}
\end{lemma}
\begin{proof}
Let $\vsf \in H^1_0(\Omega)$ be such that $\vsf_{|T} \in C^2(T)$ for all $T \in \T$. Using \eqref{optimal_system_1} and integrating by parts yields
\begin{equation*}
\label{eq:res1y}
\int_{\Omega} \nabla (\bar{\ysf} - \bar{\ysf}_{\T}) \cdot \nabla \vsf  = \sum_{T \in \T} \int_{T} \left(\fsf + \bar{\usf}\right) \vsf  + \sum_{S \in \Sides } \int_{S } \llbracket \nabla \bar{\ysf}_{\T} \cdot \nu \rrbracket  \vsf.
\end{equation*}
Since on each $T \in \T$ we have that $\vsf \in C^2(T)$, we again apply integration by parts to conclude that
\begin{equation*}
\label{eq:res2y}
\int_{\Omega} \nabla (\bar{\ysf} - \bar{\ysf}_{\T}) \cdot \nabla \vsf  = - \sum_{T \in \T} \int_{T} \Delta \vsf (\bar{\ysf} - \bar{\ysf}_{\T})   - \sum_{S \in \Sides } \int_{S} \llbracket \nabla \vsf \cdot \nu \rrbracket (\bar{\ysf} - \bar{\ysf}_{\T}).
\end{equation*}
In conclusion, since the left hand sides of the previous expressions coincide, we arrive at the identity
\begin{align}
\sum_{T \in \T} \int_{T} \left(\fsf + \bar{\usf}\right) \vsf  + \sum_{S \in \Sides } \int_{S } \llbracket \nabla \bar{\ysf}_{\T} \cdot \nu \rrbracket  \vsf\nonumber
=
&- \sum_{T \in \T} \int_{T} \Delta \vsf (\bar{\ysf} - \bar{\ysf}_{\T})   
\\
&- \sum_{S \in \Sides } \int_{S} \llbracket \nabla \vsf \cdot \nu \rrbracket (\bar{\ysf} - \bar{\ysf}_{\T}),\label{eq:error_adjointy}
\end{align}
for every $\vsf \in H^1_0(\Omega)$ such that $\vsf_{|T} \in C^2(T)$ for all $T \in \T$. We now proceed, on the basis of \eqref{eq:defofEy}, in two steps.

\noindent \framebox{Step 1.} Let $T\in\T$. We begin with the basic estimate  
\begin{equation}
\label{eq:aux_EOy}
h_{T}^{2-d/2}
\|\fsf + \bar{\usf}_{\T}\|_{L^{2}(T)}
\leq 
h_{T}^{2-d/2}
\|\mathcal{P}_{\T}\fsf + \bar{\usf}_{\T}\|_{L^{2}(T)} +
 h_{T}^{2-d/2}
\|\fsf-\mathcal{P}_{\T}\fsf\|_{L^{2}(T)}.
\end{equation}
By letting $\vsf=\beta_{T}=\left( \mathcal{P}_{\T}\fsf + \bar{\usf}_{\T}  \right)\varphi_{T}^2$ in \eqref{eq:error_adjointy}, where $\varphi_{T}$ is the standard bubble function over $T$ \cite{AObook,MR3059294}, we obtain that
\begin{equation}
\int_{T}  \left( \mathcal{P}_{\T}\fsf + \bar{\usf}_{\T}  \right) \beta_{T} 
 = 
-\int_{T} \Delta \beta_{T} (\bar{\ysf} - \bar{\ysf}_{\T}) 
- \int_{T}(\bar{\usf}-\bar{\usf}_{\T})\beta_{T} 
- \int_{T}\left( \fsf - \mathcal{P}_{\T}\fsf  \right) \beta_{T},
\label{eq:eff_auxy}
\end{equation}
since
$
\int_{S } \llbracket \nabla \beta_{T} \cdot \nu \rrbracket (\bar \ysf - \bar \ysf_{\T}) = 0
$
for all $S\in\Sides$.
We now bound each term on the right--hand side of \eqref{eq:eff_auxy} separately. Since $\Delta(\mathcal{P}_{\T}\fsf + \bar{\usf}_{\T})=0$ on $T$, we have that
$
\Delta \beta_{T} 
=4\nabla(\mathcal{P}_{\T}\fsf + \bar{\usf}_{\T})\cdot\nabla \varphi_T \varphi_T
+
2 (\mathcal{P}_{\T}\fsf + \bar{\usf}_{\T})(\varphi_T\Delta\varphi_T + \nabla\varphi_T\cdot\nabla \varphi_T).
$
This equality, the properties of the bubble function $\varphi_T$ and an inverse inequality allow us to conclude that
\begin{align*}
& \left|\int_{T} \Delta \beta_{T} (\bar{\ysf} - \bar{\ysf}_{\T})  \right| 
\\
\lesssim &
\left(
h_{T}^{d/2-1}\|\nabla(\mathcal{P}_{\T}\fsf + \bar{\usf}_{\T})\|_{L^{2}(T)} + 
h_{T}^{d/2-2}\|\mathcal{P}_{\T}\fsf + \bar{\usf}_{\T}\|_{L^{2}(T)}
\right)\|\bar{\ysf}-\bar{\ysf}_{\T}\|_{L^{\infty}(T)}
\\
\lesssim & h_{T}^{d/2-2}
\|\mathcal{P}_{\T}\fsf + \bar{\usf}_{\T}\|_{L^{2}(T)} 
\|\bar{\ysf}-\bar{\ysf}_{\T}\|_{L^{\infty}(T)}.
\end{align*}
In addition, we have that
\begin{align*}
\left| \int_{T}(\bar{\usf}-\bar{\usf}_{\T})\beta_{T}\right| 
\lesssim &
\|\bar{\usf}-\bar{\usf}_{\T}\|_{L^{2}(T)}\|\mathcal{P}_{\T}\fsf + \bar{\usf}_{\T}\|_{L^{2}(T)}
\\
\lesssim &
h_{T}^{d/2}\|\bar{\usf}-\bar{\usf}_{\T}\|_{L^{\infty}(T)}\|\mathcal{P}_{\T}\fsf + \bar{\usf}_{\T}\|_{L^{2}(T)}
\end{align*}
and
\begin{align*}
\left|\int_{T}\left( \fsf - \mathcal{P}_{\T}\fsf   \right) \beta_{T}\right|
\lesssim
\|\fsf - \mathcal{P}_{\T}\fsf\|_{L^{2}(T)}
\|\mathcal{P}_{\T}\fsf + \bar{\usf}_{\T}\|_{L^{2}(T)}.
\end{align*}
In view of the fact that
$
\|\mathcal{P}_{\T}\fsf + \bar{\usf}_{\T}\|_{L^{2}(T)}^2  
\lesssim
\int_{T}  \left( \mathcal{P}_{\T}\fsf + \bar{\usf}_{\T}  \right) \beta_{T},
$
the previous findings allow us to state that
\begin{align*}
h_{T}^{2-d/2}
\| \mathcal{P}_{\T}\fsf + \bar{\usf}_{\T}  \|_{L^{2}(T)}
\lesssim 
\|\bar{\ysf}-\bar{\ysf}_{\T}\|_{L^{\infty}(T)} 
+ h_{T}^{2}\|\bar{\usf}-\bar{\usf}_{\T}\|_{L^{\infty}(T)}
+
h_{T}^{2-d/2}\|\fsf - \mathcal{P}_{\T}\fsf\|_{L^{2}(T)}.
\end{align*}
Consequently, using \eqref{eq:aux_EOy} we conclude that
\begin{equation}\label{yeffstep1}
h_{T}^{2-d/2} \|\mathcal{P}_{\T}\fsf + \bar{\usf}_{\T} \|_{L^{2}(T)}
\lesssim
\|\bar{\ysf}-\bar{\ysf}_{\T}\|_{L^{\infty}(T)} +
 h_{T}^{2}
\|\bar{\usf}-\bar{\usf}_{\T}\|_{L^{\infty}(T)}
+
\osc_{\T}(\fsf;T).
\end{equation}

\noindent \framebox{Step 2.} Let $T \in \T$ and $S \in \Sides_T$. We proceed to control $h_T \| \llbracket   \nabla \bar{\ysf}_{\T} \cdot \nu \rrbracket  \|_{L^{\infty}(S)}$ in \eqref{eq:defofEy}. To do this, we use the property 
\begin{equation}
\label{eq:paso_vaca}
|S|\| \llbracket \nabla \bar{\ysf}_{\T} \cdot \nu \rrbracket \|_{L^{\infty}(S)} \lesssim \left|\int_S \llbracket \nabla \bar{\ysf}_{\T} \cdot \nu \rrbracket \varphi_S \right|,
\end{equation}
of $\varphi_S$, the standard bubble function over $S$ \cite{AObook,MR3059294}. We now
let $\vsf=\varphi_S$ in \eqref{eq:error_adjointy} and arrive at
\begin{align*}
  &\left| \int_S \llbracket  \nabla \bar{\ysf}_\T \cdot \nu \rrbracket \varphi_S \right|
  \\
   \leq & \sum_{T' \in \Ne_S} \int_{T'} |\fsf + \bar\usf| \varphi_S +
  \sum_{T' \in \Ne_S} \int_{T'} |\bar{\ysf}-\bar{\ysf}_\T| |\LAP \varphi_S|
 + \sum_{T' \in \Ne_S}\sum_{S' \in \Sides_{T'}}
   \int_{S'} |\bar{\ysf}-\bar{\ysf}_\T| |\llbracket \nabla \varphi_S \cdot \nu \rrbracket|
   \\
   \lesssim & \sum_{T' \in \Ne_S}|T'|^{1/2}\left( \| \fsf + \bar \usf_{\T}\|_{L^2(T')} + |T'|^{1/2} \|\bar \usf - \bar \usf_{\T} \|_{L^\infty(T')}\right)
\\  
 & + \sum_{T' \in \Ne_S}\left(h_S^{-2} |T'| + h_S^{-1} \sum_{S' \in \Sides_{T'}}|S'|\right) \| \bar{\ysf}-\bar{\ysf}_\T \|_{L^\infty(T')}.
\end{align*}
In view of the fact that $h_T |S|^{-1} \approx h_T^{2-d}$, the previous estimate combined with \eqref{eq:paso_vaca} and \eqref{yeffstep1} yields the bound
\begin{align*}
 h_T \| \llbracket  \nabla \bar{\ysf}_{\T} \cdot \nu \rrbracket  \|_{L^{\infty}(S)} &\lesssim h_T^{2} \|\bar \usf - \bar \usf_{\T} \|_{L^{\infty}(\Ne_S)} +  \| \bar \ysf -  \bar \ysf_{\T} \|_{L^{\infty}(\Ne_S)} + \osc_{\T}( \fsf;\Ne_S).
\end{align*}

We finally combine the results of Step 1 and 2 and arrive at the desired estimate \eqref{eq:efficiencyy}. This concludes the proof.
\end{proof}

Similar arguments to the ones elaborated in the proof of Lemma \ref{le:local_eff_y} allow us to conclude the following result.

\begin{lemma}[local efficiency of $\E_{\psf}$] 
In the setting of Theorem \ref{th:global_reliability} we have that
\begin{equation}
\label{eq:efficiencyp}
\E_{\psf} (\bar{\psf}_{\T},\bar{\ysf}_{\T}; T) \lesssim \|  \bar{\psf} - \bar{\psf}_{\T} \|_{L^{\infty}(\Ne_T)} + h_{T}^{2} \| \bar{\ysf} - \bar{\ysf}_\T \|_{L^\infty(\Ne_T)} 
+ \osc_{\T}(\ysf_\Omega;\Ne_T),
\end{equation}
where the hidden constant is independent of 
the size of the elements in the mesh $\T$ and $\# \T$.
\label{le:local_eff_p}
\end{lemma}

\begin{lemma}[local efficiency of $\E_{\usf}$] 
In the setting of Theorem \ref{th:global_reliability} we have that
\begin{equation}
\label{eq:efficiencyu}
\E_{\usf} (\bar{\usf}_{\T},\bar{\psf}_{\T}; T) \lesssim \left\| \bar{\usf}- \bar{\usf}_{\T}  \right\|_{L^\infty(T)} + \left\| \bar{\psf}- \bar{\psf}_{\T}  \right\|_{L^\infty(T)},
\end{equation}
where the hidden constant is independent of 
the size of the elements in the mesh $\T$ and $\# \T$.
\label{le:local_eff_u}
\end{lemma}
\begin{proof}
The estimate follows immediately from definition \eqref{eq:u} and the Lipschitz property \eqref{Lipschitz}.
\end{proof}

The results of Lemmas \ref{le:local_eff_y}, \ref{le:local_eff_p} and \ref{le:local_eff_u} immediately yield the following result upon observing that $\Omega$ is bounded.

\begin{theorem}[local and global efficiency of $\E$] 
In the setting of Theorem \ref{th:global_reliability} we have that
\begin{align}
\label{eq:efficiency_ocp}
\E (\bar{\ysf}_{\T},\bar{\psf}_{\T},\bar{\usf}_{\T}; T) \lesssim &
\|  \bar{\ysf} - \bar{\ysf}_{\T} \|_{L^{\infty}(\Ne_T)}
+\|  \bar{\psf} - \bar{\psf}_{\T} \|_{L^{\infty}(\Ne_T)}
+\left\| \bar{\usf}- \bar{\usf}_{\T}  \right\|_{L^\infty(\Ne_T)}\nonumber
\\
&+ \osc_{\T}(\fsf;\Ne_T)
+ \osc_{\T}(\ysf_\Omega;\Ne_T),
\end{align}
and
\begin{align}
\E (\bar{\ysf}_{\T},\bar{\psf}_{\T},\bar{\usf}_{\T}; \T) &\lesssim
\|(e_{\bar{\ysf}},e_{\bar{\psf}},e_{\bar{\usf}})\|_{\Omega}
+ \max_{T\in\T} \osc_{\T}(\fsf;\Ne_T)
+ \max_{T\in\T}\osc_{\T}(\ysf_\Omega;\Ne_T),
\end{align}
where the hidden constants are independent of the size of the elements in the mesh $\T$ and $\# \T$.
\label{th:localeff_ocp}
\end{theorem}

%%%%%%%%%%%%%%%%%%%%%%%%%%%%%%%%%%%%%%%%%%%%%%%%%%%%%%%%%%%%%%%%%%%%%%%%%%%%%%%%%%%%%
%%%%%%%%%%%%%%%%%%%%%%%%%%%%%%%%%%%%%%%%%%%%%%%%%%%%%%%%%%%%%%%%%%%%%%%%%%%%%%%%%%%%%
%%%%%%%%%%%%%%%%%%%%%%%%%%%%%%%%%%%%%%%%%%%%%%%%%%%%%%%%%%%%%%%%%%%%%%%%%%%%%%%%%%%%%
%%%%%%%%%%%%%%%%%%%%%%%%%%%%%%%%%%%%%%%%%%%%%%%%%%%%%%%%%%%%%%%%%%%%%%%%%%%%%%%%%%%%%
%%%%%%%%%%%%%%%%%%%%%%%%%%%%%%%%%%%%%%%%%%%%%%%%%%%%%%%%%%%%%%%%%%%%%%%%%%%%%%%%%%%%%
%%%%%%%%%%%%%%%%%%%%%%%%%%%%%%%%%%%%%%%%%%%%%%%%%%%%%%%%%%%%%%%%%%%%%%%%%%%%%%%%%%%%%

\section{Numerical example}

We illustrate the performance of the a posteriori error estimator with a numerical example. We set $\Omega=(-1,1)^{2}\setminus[0,1)\times(-1,0]$, $\asf=0$, $\bbsf=1$ and the data $\fsf$ and $\ysf_\Omega$ to be such that, in polar coordinates $(r,\theta)$ with $\theta\in[0,3\pi/2]$,
\[
\bar{\ysf}=(1-r^{2}\cos^{2}(\theta))(1-r^{2}\sin^{2}(\theta))r^{2/3}\sin(2\theta/3)
\]
and
\[
\bar{\psf}=\sin(2\pi r\cos(\theta))\sin(2\pi r\sin(\theta))r^{2/3}\sin(2\theta/3).
\]
A sequence of adaptively refined meshes was generated from an initial mesh (consisting of 12 congruent triangles) by using a maximum strategy to mark elements for refinement. The number of degrees of freedom Ndof is three times the number of vertices in the mesh. Figure \ref{Fig:1} shows the results and we can observe that, once the mesh has been sufficiently refined, the error and estimator converge at the optimal rate.

\begin{figure}[!htbp]
  \begin{center}
  \psfrag{total Linfinity error}{\huge $\|(e_{\bar{\ysf}},e_{\bar{\psf}},e_{\bar{\usf}})\|_{\Omega}$}
  \psfrag{total Linfinity estimator}{\huge $\E(\bar{\ysf}_{\T},\bar{\psf}_{\T},\bar{\usf}_{\T}; \T)$}
  \psfrag{Optimal rate}{\huge Ndof$^{-1}$}
  \psfrag{Ndof}{\huge Ndof}
    \scalebox{0.42}{\includegraphics{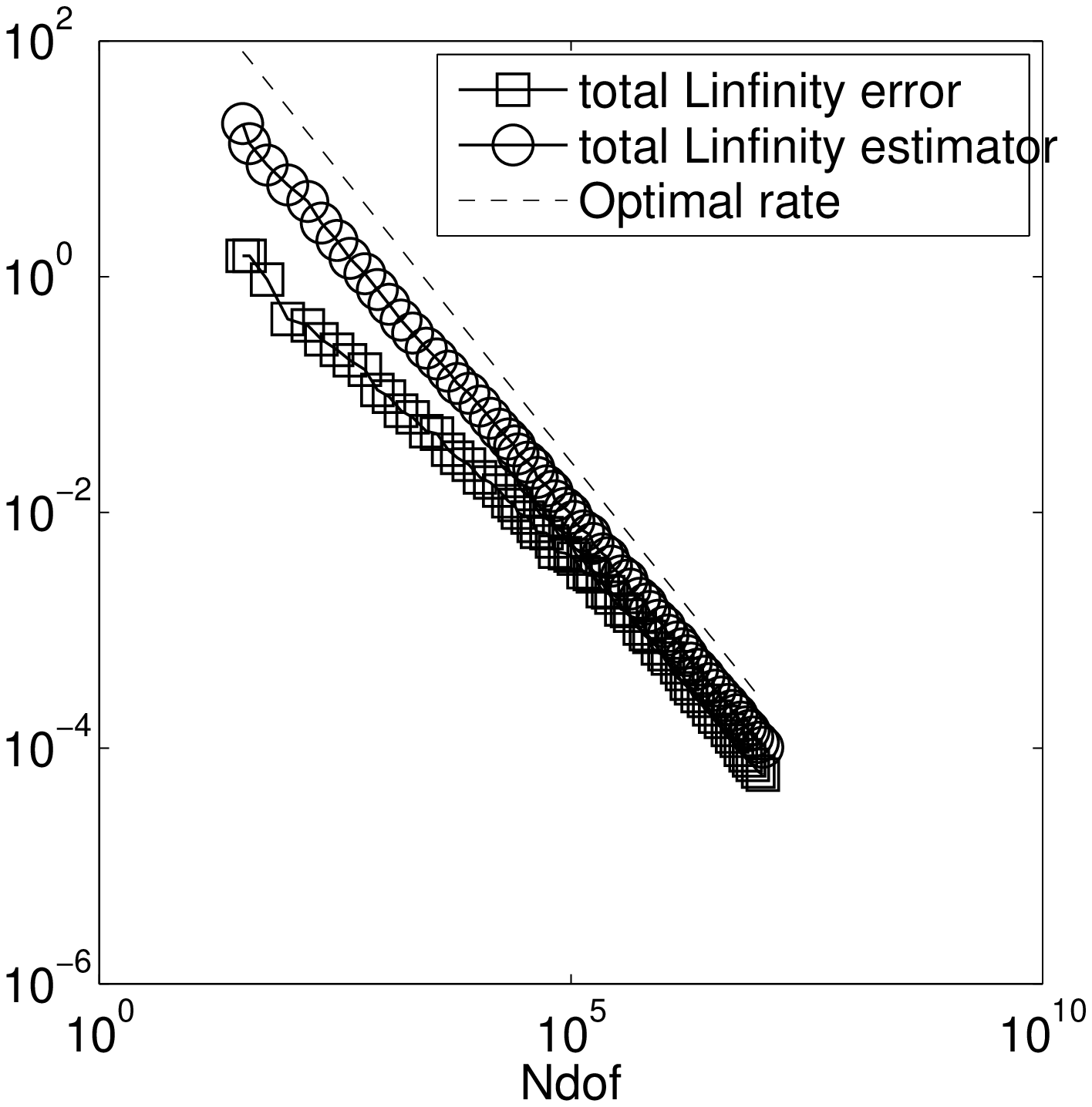}}
    \scalebox{0.48}{\includegraphics{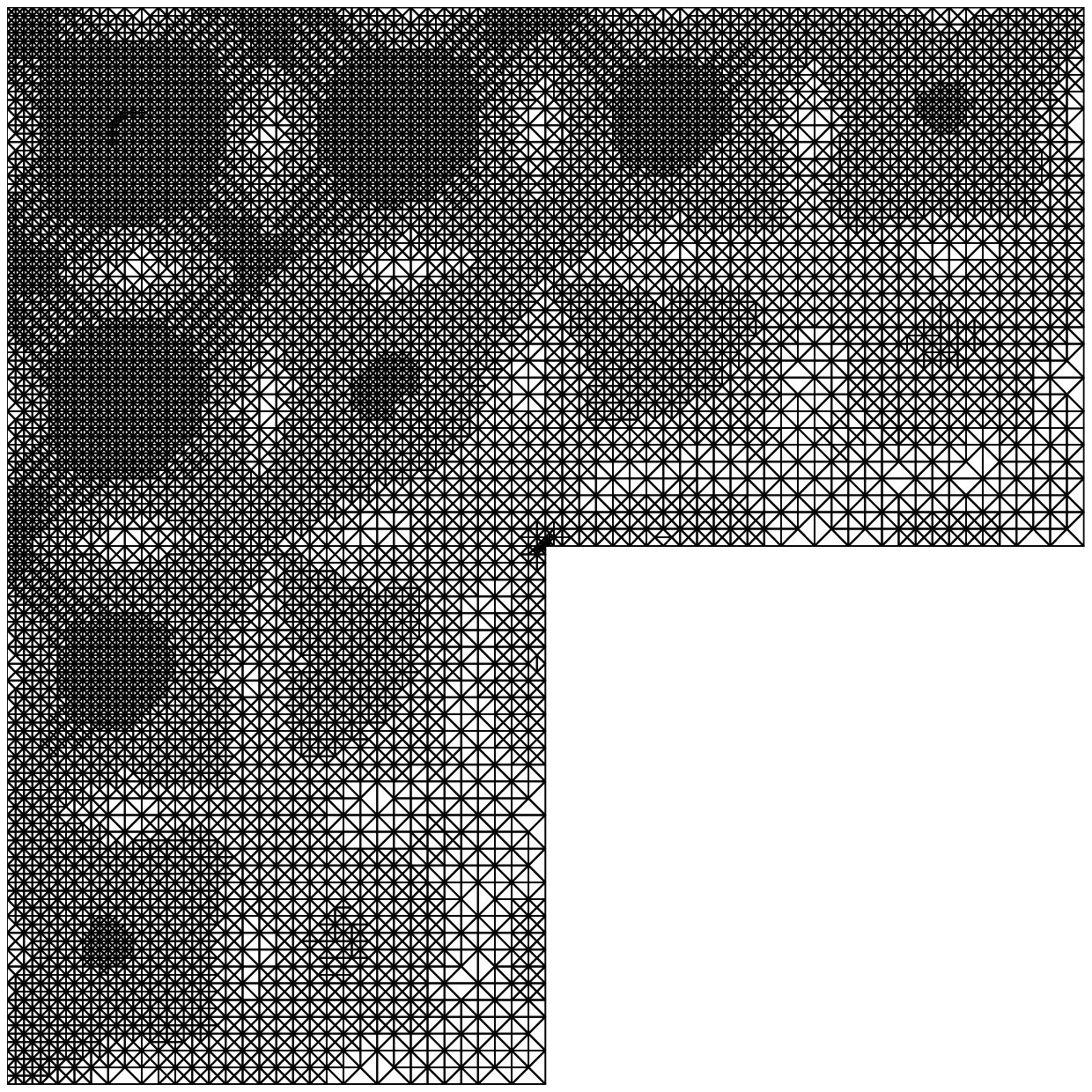}}
  \end{center}
\caption{The error $\|(e_{\bar{\ysf}},e_{\bar{\psf}},e_{\bar{\usf}})\|_{\Omega}$ and the estimator $\E(\bar{\ysf}_{\T},\bar{\psf}_{\T},\bar{\usf}_{\T}; \T)$ (left) and the 24th adaptively refined mesh (right).}
\label{Fig:1}
\end{figure}

\begin{acknowledgements}
A. Allendes was supported in part by CONICYT through FONDECYT project 1170579. E. Ot\'arola was supported in part by CONICYT through FONDECYT project 3160201. A. J. Salgado was supported in part by NSF grant DMS-1418784.  
\end{acknowledgements}

\bibliographystyle{spmpsci} 
\bibliography{bibliography}

\end{document}